\documentclass{article}
\usepackage{amsmath}
\usepackage{amsthm}
\usepackage{amsfonts}
\usepackage{amsfonts}

\usepackage[all]{xy}

\usepackage{url}

\def\Br{\operatorname{Br}}

\newtheorem{theorem}{Theorem}
\newtheorem{lemma}[theorem]{Lemma}
\newtheorem{corollary}[theorem]{Corollary}

\begin{document}

\title{Data and homotopy types}
\author{J.F. Jardine}

\maketitle

\begin{abstract}
  This paper presents explicit assumptions for the existence of interleaving homotopy equivalences of both Vietoris-Rips and Lesnick complexes associated to an inclusion of data sets. Consequences of these assumptions are investigated on the space level, and for corresponding hierarchies of clusters and their sub-posets of branch points. Hierarchy posets and  branch point posets  admit a calculus of least upper bounds, which is used to show that the map of branch points associated to the inclusion of data sets is a controlled homotopy equivalence.
  \end{abstract}

\section*{Introduction}

This paper is a discussion of
homotopy theoretic phenomena that arise in connection with inclusions $X \subset Y \subset \mathbb{R}^{n}$ of data sets in topological data analysis. 
The manuscript is not in final form, and comments are welcome.
\smallskip

Suppose that $r > 0$, and say that $X$ is $r$-dense in $Y$ if for every point $y$ in $Y$ there is an $x$ in $X$ such that the distance $d(y,x) < r$ in the ambient metric space.

The first result in this direction (Corollary \ref{cor 2x} below) implies that, under the $r$-density assumption, the induced map $V_{s}(X) \to V_{s}(Y)$ of Vietoris-Rips complexes is a homotopy equivalence if $2r < t-s$, where $t$ is the smallest number such that $V_{t}(X) \ne V_{s}(X)$.

This result can be extended to the assertion that the inclusion $L_{s,k}(X) \to L_{s,k}(Y)$ of Lesnick complexes (with fixed density parameter $k$) is a homotopy equivalence if every point $y$ in the ``configuration space'' $Y^{k+1}_{dis}$ of $k+1$ distinct elements of $Y$ has an element $x \in X^{k+1}_{dis}$ such that $d(y,x) < r$ in $\mathbb{R}^{n(k+1)}$. This result appears as Corollary \ref{cor 4x} in this paper.

Corollary \ref{cor 2x} and Corollary \ref{cor 4x} follow from Theorem \ref{th 1x} and Theorem \ref{th 3x}, respectively. Both theorems are ``interleaving'' homotopy type results (see also \cite{BlumLes}) that follow from the respective $r$-density assumptions, by methods that amount to manipulation of barycentric subdivisions. Theorem \ref{th 1x} is a special case of Theorem \ref{th 3x} (it is the case $k=0$), but Theorem \ref{th 1x} was found initially, and its proof has a certain clarity in isolation.

For a fixed $k$, the sets $\pi_{0}L_{s,k}(X)$ of path components, as $s$ varies, define a tree $\Gamma_{k}(X)$ with elements $(s,[x])$, $[x] \in \pi_{0}L_{s,k}(X)$. The inclusion $X \subset Y$ defines a poset morphism $\Gamma_{k}(X) \to \Gamma_{k}(Y)$.

The tree $\Gamma_{k}(X)$ is the object studied by the HDBSCAN clustering algorithm, while the individual sets of clusters $\pi_{0}L_{s,k}(X)$ are the objects of interest for the DBSCAN algorithm.

The poset $\Gamma_{k}(X)$ has a subobject $\Br_{k}(X)$ whose elements are the branch points of $\Gamma_{k}(X)$, suitably defined --- see Section 2. The branch points of $\Gamma_{k}(X)$ are in one to one correspondence with the stable components for $\Gamma_{k}(X)$ that are defined in \cite{clusters9}, in the sense that every such stable component starts at a unique branch point. Thus, we can (and do) replace the stable component discussion of \cite{clusters9} with the branch point poset $\Br_{k}(X)$, and make particular use of its ordering.

The tree $\Gamma_{k}(X)$ has least upper bounds, and these restrict to least upper bounds for the subobject $\Br_{k}(X)$ of branch points. This notion of least upper bounds is an extension of and a potential replacement for the distance function that is introduced by Carlsson and M\'emoli \cite{CM-2010B} in their description of the ultrametric structure on a data set $X$ that arises from the single linkage cluster hierarchy. The Carlsson-M\'emoli theory does not apply in general to the tree $\Gamma_{k}(X)$, because the vertex sets of the Lesnick complexes $L_{k,s}(X)$ vary with changes of the parameter $s$.

The calculus of least upper bounds and its relation with branch points is described in Lemmas \ref{lem 7x}--\ref{lem 12x} below.

The branch point tree $\Br_{k}(X)$ can be thought of as a highly compressed version of the hierarchy $\Gamma_{k}(X)$ that is produced by the HDBSCAN algorithm.

The inclusion $\Br_{k}(X) \subset \Gamma_{k}(X)$ is a homotopy equivalence of posets, where the homotopy inverse is defined by taking the maximal branch point $(s_{0},[x_{0}]) \leq (s,[x])$ below $(s,[x])$ for each object of $\Gamma_{k}(X)$. The existence of the maximal branch point below an object $(s,[x])$ is a consequence of Lemma \ref{lem 12x}.

The poset map $\Gamma_{k}(X) \to \Gamma_{k}(Y)$ defines a poset map $i_{\ast}: \Br_{k}(X) \to \Br_{k}(Y)$, via the homotopy equivalences for the data sets $X$ and $Y$ of the last paragraph.

The configuration space $r$-density assumption for Theorem \ref{th 3x} imples that there is a poset morphism $\theta_{\ast}: \Br_{k}(Y) \to \Br_{k}(X)$ that is induced by a morphism $(s,[x]) \to (s+2r,[\theta(x)])$, and that there are homotopies of poset maps $i_{\ast} \cdot \theta_{\ast} \simeq s_{\ast}$ and $\theta_{\ast} \cdot i_{\ast} \simeq s_{\ast}$, where $s_{\ast}$ is defined by the shift operator $(s,[x]) \mapsto (s+2r,[x])$ on $Y$ and $X$, respectively. These homotopies are ``bounded'' (or controlled) by the number $2r$.

In good circumstances, a branch point $(s,[x])$ is the maximal branch point below the shift $(s+2r,[x])$, and the inequalities defining the homotopies of the last paragraph become equalities in that case.
\medskip

In summary, the poset morphism $\Br_{k}(X) \to \Br_{k}(Y)$ that is induced by an inclusion of data sets $X \subset Y$ has a homotopy theoretic character, and is measurably close to a homotopy equivalence if every sufficiently large group of distinct points of $Y$ is close to a corresponding group of distinct points for the smaller data set $X$. Such a statement amounts to a stability result for hierarchies of branch points, albeit not in traditional terms.

\tableofcontents


\section{Homotopy types}

Suppose given (finite) data sets $X \subset Y \subset \mathbb{R}^{n}$. Suppose that $r > 0$.

Say that $X$ is {\bf $r$-dense} in $Y$, if for all $y \in Y$ there is an $x \in X$ such that $d(y,x) < r$.
\medskip

Suppose that $s \geq 0$. Recall that $V_{s}(X)$ is the simplicial complex with simplices $(x_{0}, \dots ,x_{n})$ with $x_{i} \in X$ and $d(x_{i},x_{j}) \leq s$.

Then we have
\begin{equation*}
  X = V_{0}(X) \subset V_{s}(X) \subset V_{t}(X) \subset \dots \subset V_{R}(X) = \Delta^{X}
\end{equation*}
for $0 \leq s < t \leq R$, and for $R$ sufficiently large, where $\Delta^{X} := \Delta^{N}$ and $N = \vert X \vert -1$.

The inclusion $i: X \subset Y$ induces a map of systems of simplicial complexes
$i: V_{s}(X) \subset V_{s}(Y)$.

The data sets $X$ and $Y$ are finite, so there is a finite string of parameter values
\begin{equation*}
  0=s_{0} < s_{1} < \dots < s_{r}, 
\end{equation*}
consisting of the distances between elements of $Y$. This includes the list of distances between elements of $X$. I say that the $s_{i}$ are the {\bf phase-change} numbers.

\begin{theorem}\label{th 1x}
Suppose that $X \subset Y \subset \mathbb{R}^{n}$, and that $X$ is $r$-dense in $Y$.Then there is a homotopy commutative diagram
    \begin{equation*}
      \xymatrix{
        V_{s}(X) \ar[r]^{j} \ar[d]_{i} & V_{s+2r}(X) \ar[d]^{i} \\
        V_{s}(Y) \ar[r]_{j} \ar[ur]^{\theta} & V_{s+2r}(Y)
      }
      \end{equation*}
    in which the upper triangle commutes.
\end{theorem}
 
\begin{proof}
            Define a function $\theta: Y \to X$ by specifying $\theta(x) = x$ for $x \in X$. For $y \in Y-X$ find $x \in X$ such that $d(x,y) < r$ and set $\theta(y) =x$.

        If $(y_{0}, \dots ,y_{n})$ is a simplex of $V_{s}(Y)$,
        \begin{equation*}
          d(\theta(y_{i}),\theta(y_{j})) \leq d(\theta(y_{i}),y_{i})
          + d(y_{i},y_{j}) + d(y_{j},\theta(y_{j})) \leq r + s + r.
          \end{equation*}
        It follows that $\theta$ induces a simplicial complex map $\theta: V_{s}(Y) \to V_{s+2r}(X)$, such that the upper triangle commutes.

        For the simplex $(y_{0}, \dots ,y_{n})$ of $V_{s}(Y)$, the string of elements
        \begin{equation*}
          (y_{0}, \dots,y_{n},\theta(y_{0}), \dots ,\theta(y_{n}))
        \end{equation*}
        defines a simplex of $V_{s+2r}(Y)$, since
        \begin{equation*}
          d(y_{i},\theta(y_{j}) \leq d(y_{i},y_{j}) + d(y_{j},\theta(y_{j}) \leq s+r.
        \end{equation*}
        Set
        \begin{equation*}
          \gamma(y_{0}, \dots ,y_{n}) = (y_{0}, \dots,y_{n},\theta(y_{0}), \dots ,\theta(y_{n})).
        \end{equation*}
        This assignment defines a morphism $\gamma: NV_{s}(Y) \to NV_{s+2r}(Y)$ of posets of non-degenerate simplices, and there are homotopies (natural transformations)
        \begin{equation*}
          j \to  \gamma \leftarrow i\cdot\theta
        \end{equation*}
        which are defined by face inclusions.
\end{proof}

        \begin{corollary}\label{cor 2x}
    Suppose that $X \subset Y \subset \mathbb{R}^{n}$, and that $X$ is $r$-dense in $Y$. Suppose that $2r < s_{i+1}-s_{i}$. Then the map
    \begin{equation*}
      i: V_{s_{i}}(X) \to V_{s_{i}}(Y)
    \end{equation*}
    is a weak homotopy equivalence.
\end{corollary}

\begin{proof}
        In the homotopy commutative diagram
\begin{equation*}
      \xymatrix{
        V_{s_{i}}(X) \ar[r]^{j} \ar[d]_{i} & V_{s_{i}+2r}(X) \ar[d]^{i} \\
        V_{s_{i}}(Y) \ar[r]_{j} \ar[ur]^{\theta} & V_{s_{i}+2r}(Y)
      }
      \end{equation*}
the horizontal morphisms $j$ are isomorphisms (identities), and so $V_{s_{i}}(X)$ is a deformation retract of $V_{s_{i}}(Y)$.
\end{proof}

Corollary \ref{cor 2x} has consequences for both persistent homology and clustering.
\medskip

Suppose that $k$ is a non-negative integer. The Lesnick subcomplex $L_{s,k}(X)$ is the full subcomplex of $V_{s}(X)$ on those vertices $x$ for which there are at least $k$ distinct vertices $x_{i} \ne x$ such that $d(x,x_{i}) \leq s$ \cite{HMc}, \cite{clusters9}, \cite{LW2}.

A simplex $\sigma = (y_{0}, \dots ,y_{n})$ of $V_{s}(X)$ is in $L_{s,k}(X)$ if and only if each vertex $y_{i}$ has at least $k$ distinct neighbours in $V_{s}(X)$ --- this is the meaning of the assertion that $L_{s,k}(X)$ is a full subcomplex of $V_{s}(X)$.

There is an array of subcomplexes
\begin{equation*}
  \xymatrix{
    V_{s}(X) \ar[r] & V_{t}(X) \\
    L_{s,k}(X) \ar[r] \ar[u] & L_{t,k}(X) \ar[u] \\
    L_{s,k+1}(X) \ar[r] \ar[u] & L_{t,k+1}(X) \ar[u]
  }
  \end{equation*}
We have the following observations:
\begin{itemize}
\item[1)] $L_{s,0}(X) = V_{s}(X)$.
\item[2)] $L_{s,k}(X)$ could be empty for small $s$ and large $k$. In general, for $s \leq t$, $L_{s,k}(X)$ and $L_{t,k}(X)$ may not have the same vertices.
  \item[3)] Every inclusion $i: X \subset Y \subset \mathbb{R}^{n}$ induces maps $i: L_{s,k}(X) \to L_{s,k}(Y)$ which are natural in $s$ and $k$.
\end{itemize}

I say that $s$ is a {\bf spatial parameter} and that $k$ is a {\bf density parameter} (also a {\bf valence}, or degree). Lesnick says \cite{LW2} that $\{ L_{s,k}(X) \}$ is the {\bf degree Rips filtration} of the system $\{ V_{s}(X) \}$.

\medskip

Write $X^{k+1}_{dis}$ for the set of $k+1$ distinct points of $X$, and think of it as a subobject of $(\mathbb{R}^{n})^{k+1}$. 

\begin{theorem}\label{th 3x}
        Suppose that $X \subset Y \subset \mathbb{R}^{n}$ and that $X^{k+1}_{dis}$ is $r$-dense in $Y^{k+1}_{dis}$ and that $L_{s,k}(Y) \ne \emptyset$. Then there is a homotopy commutative diagram
        \begin{equation*}
          \xymatrix{
          L_{s,k}(X) \ar[r]^{j} \ar[d]_{i} & L_{s+2r,k}(X) \ar[d]^{i} \\
          L_{s,k}(Y) \ar[r]_{j} \ar[ur]^{\theta} & L_{s+2r,k}(Y) 
          }
        \end{equation*}
        in which the upper triangle commutes in the usual sense.
\end{theorem}

\begin{proof}
            Suppose that $y \in L_{s,k}(X)_{0} - L_{s,k}(X)_{0}$. Then there are $k$ points $y_{1},\dots ,y_{k}$ of $Y$, distinct from $y$ such that $d(y,y_{i}) < s$. There is a $(k+1)$-tuple $(x_{0},x_{1}, \dots ,x_{k})$ such that
            \begin{equation*}
              d((x_{0}, \dots ,x_{k}),(y,y_{1}, \dots ,y_{k})) < r,
            \end{equation*}
            by assumption. Then $d(y,x_{0}) <r$, $d(y_{i},x_{i}) < r$, and so $d(x_{0},x_{i}) < s+2r$, so that $x_{0} \in L_{s+2r,k}(X)$.  Set $\theta(y) = x_{0}$, and observe that $d(y,\theta(y)) < r$.

            If $(y_{0}, \dots ,y_{p})$ is a simplex of $L_{s,k}(Y)$ then $(\theta(y_{0}), \dots ,\theta(y_{p})$ is a simplex of $L_{s+2r,k}(X)$, as is the string
              \begin{equation*}
                (y_{0}, \dots ,y_{p},\theta(y_{0}), \dots ,\theta(y_{p})).
              \end{equation*}
              Finish according to the method of proof for Theorem \ref{th 1x}.
\end{proof}

\begin{corollary}\label{cor 4x}
  Suppose that $X \subset Y \subset \mathbb{R}^{n}$ and that $X^{k+1}_{dis}$ is $r$-dense in $Y^{k+1}_{dis}$. Suppose that $2r < s_{i+1}-s_{i}$. Then the inclusion
$i: L_{s_{i},k}(X) \to L_{s_{i},k}(Y)$ is a weak homotopy equivalence.
\end{corollary}

\begin{lemma}
    Suppose that $X^{k+1}_{dis}$ is $r$-dense in $Y^{k+1}_{dis}$, and that $Y^{k+1}_{dis} \ne \emptyset$. Then $X^{k}_{dis}$ is $r$-dense in $Y^{k}_{dis}$.
\end{lemma}

\begin{proof}
    Suppose that $\{ y_{0}, \dots ,y_{k-1} \}$ is a set of $k$ distinct points of $Y$. Then there is a $y_{k} \in Y$ which is distinct from the $y_{i}$, so that $(y_{0},y_{1}, \dots ,y_{k})$ is a $(k+1)$-tuple of distinct points of $Y$. There is a $(k+1)$-tuple $(x_{0},\dots ,x_{k})$ of distinct points of $X$ such that
    \begin{equation*}
      d((y_{0}, \dots ,y_{k-1},y_{k}),(x_{0}, \dots ,x_{k-1},x_{k})) < r.
      \end{equation*}
        It follows that
    \begin{equation*}
      d((y_{0}, \dots ,y_{k-1}),(x_{0}, \dots ,x_{k-1})) < r.
    \end{equation*}
\end{proof}    

\begin{corollary}
        Suppose that $X \subset Y \subset \mathbb{R}^{n}$ and that $X^{k+1}_{dis}$ is $r$-dense in $Y^{k+1}_{dis}$. Suppose that $2r < s_{i+1}-s_{i}$. Then the inclusion
$i: L_{s_{i},p}(X) \to L_{s_{i},p}(Y)$ is a weak homotopy equivalence for $0 \leq p \leq k$.
\end{corollary}    

\section{Branch points and upper bounds}

Fix the density $k$ and suppose that $L_{s,k}(X) \ne \emptyset$ for $s$ sufficiently large. Apply the path component functor to the $L_{s,k}(X)$, to get a diagram of functions
\begin{equation*}
  \dots \to \pi_{0}L_{s,k}(X) \to \pi_{0}L_{t,k} \to \dots
\end{equation*}

There is a graph $\Gamma_{k}(X) := \Gamma(\pi_{0}L_{\ast,k}(X))$ with vertices $(s,[x])$ with $[x] \in \pi_{0}L_{s,k}(X)$, and edges $(s,[x]) \to (t,[x])$ with $s \leq t$. This graph underlies a contractible poset, and is therefore a tree (or hierarchy).

To reflect the poset structure of $\Gamma_{k}(X)$, I write the morphisms of $\Gamma_{k}(X)$ as relations $(s,[x]) \leq (t,[y])$. The existence of such a relation means that  $[x] = [y] \in \pi_{0}L_{t,k}(X)$, or that the image of $[x] \in \pi_{0}L_{s,k}(X)$ under the induced function $\pi_{0}L_{s,k}(X) \to \pi_{0}L_{t,k}(X)$ is $[y]$.
\medskip

\noindent
    {\bf Remarks}:\ 1)\ Partitions of $X$ given by the set $\pi_{0}V_{s}(X)$ are standard {\bf clusters}. The tree $\Gamma_{0}(X)=\Gamma(V_{\ast}(X))$ defines a {\bf hierarchical clustering} (similar to, but not the same as single linkage clustering).
    \smallskip

    \noindent
    2)\ The set $\pi_{0}L_{s,k}(X)$ gives a partitioning of the set of elements of $X$ having at least $k$ neighbours of distance $\leq s$, which is the subject of the {\bf DBSCAN} algorithm. The tree $\Gamma_{k}(X)=\Gamma(\pi_{0}L_{\ast,k}(X))$ is the basis of the {\bf HDBSCAN} algorithm.
    \medskip

    A {\it branch point} in the tree $\Gamma_{k}(X)$ is a vertex $(t,[x])$ such that either of following two conditions hold:
    \begin{itemize}
    \item[1)] there is an $s_{0} < t$ such that for all $s_{0} \leq s < t$ there are distinct vertices $(s,[x_{0}])$ and $(s,[x_{1}])$ with $(s,[x_{0}]) \leq (t,[x])$ and $(s,[x_{1}]) \leq (t,[x])$, or
    \item[2)] there is no relation $(s,[y]) \leq (t,[x])$ with $s < t$.
\end{itemize}
The second condition means that a representing vertex $x$ is not a vertex of $L_{s,k}(X)$ for $s < t$.    
Write $\Br_{k}(X)$ for the set of branch points $(s,[x])$ in $\Gamma_{k}(X)$.

Every branch point $(s,[x])$ of $\Gamma_{k}(X)$ has $s=s_{i}$, where $s_{i}$ is a phase change number for $X$.

The branch point poset $\Br_{k}(X)$ is a tree, because the element $(s_{r},[x])$ corresponding to the highest phase change number $s_{r}$ is maximal.

The set of branch points $\Br_{k}(X)$ inherits a partial ordering from the poset $\Gamma_{k}(X)$, and the inclusion $\Br_{k}(X) \subset \Gamma_{k}(X)$ of the set of branch points defines a monomorphism of trees.

\medskip

Suppose that $(s,[x])$ and $(t,[y])$ are vertices of the graph $\Gamma_{k}(X)$. There is a vertex $(v,[w])$ such that $(s,[x]) \leq (v,[w])$ and $(t,[y]) \leq (v,[w])$. The two relations mean, among other things, that $[x]=[z]=[y]$ in $\pi_{0}L_{v,k}(X)$.

It follows that there is a unique smallest vertex $(u,[z])$ which is an upper bound for both $(s,[x])$ and $(t,[y])$. The number $u$ is the smallest parameter $v$ such that $[x]=[y]$ in $\pi_{0}L_{u,k}(X)$, and so $[z]=[x]=[y]$.

I write
\begin{equation*}
  (s,[x]) \cup (t,[y]) = (u,[z]).
  \end{equation*}
The vertex $(u,[z])$ is the {\bf least upper bound} (or join) of $(s,[x])$ and $(t,[y])$.

Every finite collection of points $(s_{1},[x_{1}]), \dots ,(s_{p},[x_{p}])$ has a least upper bound
\begin{equation*}
  (s_{1},[x_{1}]) \cup \dots \cup (s_{p},[x_{p}])
\end{equation*}
in $\Gamma_{k}(X)$.

\begin{lemma}\label{lem 7x}
  The least upper bound of branch points $(u,[z])$ of $(s,[x])$ and $(t,[y])$ is a branch point.
\end{lemma}

\begin{proof}
  For numbers $s,t<v<u$, $(v,[x])$ and $(v,[y])$ are distinct, because $(u,[z])$ is a least upper bound.

  Otherwise, $s=u$ or $t=u$, in which case $(u,[z]) = (s,[x])$ or $(u,[z])=(t,[y])$, respectively.
  \end{proof}

It follows from Lemma \ref{lem 7x} that any two branch points $(s,[x])$ and $(t,[y])$ have a least upper bound in $\Br_{k}(X)$, and that the poset inclusion $\alpha: \Br_{k}(X) \to \Gamma_{k}(X)$ preserves least upper bounds.
\medskip

\noindent
{\bf Remark}:\ The poset $\Gamma_{k}(X)$ also has greatest lower bounds (or meets). The greatest lower bound
\begin{equation*}
  (t_{1},[y_{1}]) \cap \dots \cap (t_{r},[y_{r}])
\end{equation*}
  is the least upper bound of all $(s,[x])$ such that $(s,[x]) \leq (t_{j},[y_{j}])$ for all $j$.
  \medskip
  
  We have the following triviality:

  \begin{lemma}\label{lem 8x}
    Suppose that $(s_{1},[x_{1}]), (s_{2},[x_{2}])$ and $(s_{3},[x_{3}])$ are vertices of $\Gamma_{k}(X)$. Then
    \begin{equation*}
      (s_{1},[x_{1}]) \cup (s_{3},[x_{3}]) \leq ((s_{1},[x_{1}]) \cup (s_{2},[x_{2}])) \cup ((s_{2},[x_{2}]) \cup (s_{3},[x_{3}])).
      \end{equation*}
  \end{lemma}

  \begin{proof}
    We have the identity
    \begin{equation*}
      ((s_{1},[x_{1}]) \cup (s_{2},[x_{2}])) \cup ((s_{2},[x_{2}]) \cup (s_{3},[x_{3}])) = (s_{1},[x_{1}]) \cup (s_{2},[x_{2}]) \cup (s_{3},[x_{3}]),
      \end{equation*}
    and then
    \begin{equation*}
      (s_{1},[x_{1}]) \cup (s_{3},[x_{3}]) \leq (s_{1},[x_{1}]) \cup (s_{2},[x_{2}]) \cup (s_{3},[x_{3}]).
      \end{equation*}
    \end{proof}

  Carlsson and M\'emoli \cite{CM-2010B} define an ultrametric $d$ on $X=V_{0}(X)$, for which they say that $d(x,y)=s$, where $s$ is the minimum parameter value such that $[x]=[y] \in
 \pi_{0}V_{s}(X)$.

 Suppose given $[x]$ and $[y]$ in $\pi_{0}L_{s,k}(X)$ (equivalently, points $(s,[x])$ and $(s,[y])$ in $\Gamma_{k}(X)$). Write $d([x],[y]) = u-s$, where $(s,[x]) \cup (s,[y]) = (u,[w])$.

 \begin{lemma}\label{lem 9x}
   Given $[x],[y]$ and $[z]$ in $\pi_{0}L_{s,k}(X)$, we have a relation
   \begin{equation*}
     d([x],[z]) \leq \max \{ d([x],[y]),d([y],[z]) \}.
\end{equation*}
   \end{lemma}

 \begin{proof}
   Suppose that $(s,[x]) \cup (s,[y]) \cup (s,[z]) = (v,[w])$. Then
   \begin{equation*}
     v -s = \max \{ d([x],[y]),d([y],[z]) \}
   \end{equation*}
   and
   \begin{equation*}
     d([x],[z]) \leq v-s
   \end{equation*}
   by Lemma \ref{lem 8x}.
 \end{proof}

 
 \begin{corollary}\label{cor 10x}
   The function
   \begin{equation*}
     d: \pi_{0}L_{s,k}(X) \times \pi_{0}L_{s,k}(X) \to \mathbb{R}_{\geq 0}
     \end{equation*}
of Lemma \ref{lem 9x} gives the set $\pi_{0}L_{s,k}(X)$ the structure of an ultrametric space.
   \end{corollary}

 \noindent
     {\bf Remark}:\ One could define a ``distance'' function $d$ on the full set of points of $\Gamma_{k}(X)$ by setting
     \begin{equation*}
       d((s,[x]),(t,[y])) = \max \{ u-s,u-t \},
     \end{equation*}
     where $(s,[x]) \cup (t,[y]) = (u,[z])$.

The ultrametric property of Lemma \ref{lem 9x} fails for the points $(s,[x]),(t,[x])$ and $(u,[x])$ where $s < t <u$, since it is not the case that $u-s \leq \max \{ t-s,u-t \}$

\begin{lemma}\label{lem 11x}
Every vertex $(s,[x])$ of $\Gamma_{k}(X)$ has a unique largest branch point $(s_{0},[x_{0}])$ such that $(s_{0},[x_{0}]) \leq (s,[x])$.
\end{lemma}

\begin{proof}
The least upper bound of the finite list of the branch points $(t,[y])$ such that $(t,[y]) \leq (s,[x])$ is a branch point, by Lemma \ref{lem 7x}.
\end{proof}

In the situation described by Lemma \ref{lem 11x}, I say that $(s_{0},[x_{0}])$ is the {\bf maximal branch point below} $(s,[x])$.

\begin{lemma}\label{lem 12x}
  Suppose that $(s_{0},[x_{0}])$ and $(t_{0},[y_{0}])$ are maximal branch points below the points $(s,[x])$ and $(t,[y])$, respectively.

  Then $(s_{0},[x_{0}]) \cup (t_{0},[y_{0}])$ is the maximal branch point below $(s,[x]) \cup (t,[y])$.
\end{lemma}

\begin{proof}
  Suppose that $s \leq t$.
  
We have 
\begin{equation*}
(s_{0},[x_{0}]) \cup (t_{0},[y_{0}]) \leq (s,[x]) \cup (t,[y]).
\end{equation*}
and $(s_{0},[x_{0}]) \cup (t_{0},[y_{0}])$ is a branch point by Lemma \ref{lem 7x}.
Write
\begin{equation*}
(v,[z]) = (s_{0},[x_{0}]) \cup (t_{0},[y_{0}]). 
\end{equation*}

Suppose that $v \leq t$.
Then $(t_{0},[y_{0}]) \leq (t,[y])$ and $(t_{0},[y_{0}]) \leq (v,[z])$, so that $(v,[z]) \leq (t,[y])$ since $v \leq t$. Also, $(s_{0},[x_{0}]) \leq (s,[x])$ and $(s_{0},[x_{0}]) \leq (v,[z]) \leq (t,[y])$ so that $(s,[x]) \leq (t,[y])$. Then $(s_{0},[x_{0}]) \leq (t_{0},[y_{0}])$ by maximality, and it follows that
\begin{equation*}
  (s_{0},[x_{0}]) \cup (t_{0},[y_{0}]) = (t_{0},[y_{0}])
\end{equation*}
is the maximal branch point of
\begin{equation*}
  (s,[x]) \cup (t,[y]) = (t,[y])
\end{equation*}

  Suppose that $v >t$. Then $(s,[x]) = (s,[x_{0}]) \leq (v,[z])$ and $(t,[y]) = (t,[y_{0}]) \leq (v,[z])$ because $s \leq t <v$, so that
\begin{equation*}
  (s,[x]) \cup (t,[y]) \leq (s_{0},[x_{0}]) \cup (t_{0},[y_{0}]),
\end{equation*}
Thus, $(s_{0},[x_{0}]) \cup (t_{0},[y_{0}]) = (s,[x]) \cup (t,[y])$ is a branch point, 
by Lemma \ref{lem 7x}.
\end{proof}

The poset inclusion $\alpha: \Br_{k}(X) \to \Gamma_{k}(X)$ has an inverse
\begin{equation*}
  max: \Gamma_{k}(X) \to \Br_{k}(X),
\end{equation*}
  up to homotopy.

In effect, Lemma \ref{lem 11x} implies that every vertex $(s,[x])$ of $\Gamma_{k}(X)$ has a unique maximal branch point $(s_{0},[x_{0}])$ such that $(s_{0},[x_{0}]) \leq (s,[x])$. Set
\begin{equation*}
  max(s,[x]) = (s_{0},[x_{0}]).
\end{equation*}
The maximality condition implies that $max$ preserves the ordering. The composite
$max \cdot \alpha$ is the identity on $\Br_{k}(X)$, and the relations $(s_{0},[x_{0}]) \leq (s,x)$ define a homotopy $max \cdot \alpha \leq 1$.
\medskip

Return to the inclusion $i: X \subset Y \subset \mathbb{R}^{n}$ of finite data sets. Suppose that $X^{k+1}_{dis}$ is $r$-dense in $Y^{k+1}_{dis}$ and that $L_{s,k}(Y)$ is non-empty, as in the statement of Theorem \ref{th 3x}.

Write $i_{\ast}: \Br_{k}(X) \to \Br_{k}(Y)$ for the composite
\begin{equation*}
  \Br_{k}(X) \xrightarrow{\alpha} \Gamma_{k}(X) \xrightarrow{i_{\ast}} \Gamma_{k}(Y) \xrightarrow{max} \Br_{k}(Y)
  \end{equation*}
This map takes a branch point $(s,[x])$ to the maximal branch point below $(s,[i(x)])$. The map $i_{\ast}$ preserves least upper bounds by Lemma \ref{lem 7x}.

Poset morphisms $\theta_{\ast}: \Br_{k}(Y) \to \Br_{k}(X)$ and $s_{\ast}: \Br_{k}(X) \to \Br_{k}(X)$ are similarly defined, by the poset morphism $\theta: \Gamma_{k}(Y) \to \Gamma_{k}(X)$ with $(t,[y]) \mapsto (t+2r,[\theta(y)])$, and the shift morphism $s: \Gamma_{k}(X) \to \Gamma_{k}(X)$ with $(s,[x]) \mapsto (s+2r,[x])$.

The construction of the poset map $i_{\ast}: \Br_{k}(X) \to \Br_{k}(Y)$ is not functorial in maps of the form $X \to Y$, but it is functorial up to coherent homotopy.

Similarly, the map $i_{\ast}: \Br_{k}(X) \to \Br_{k}(Y)$ only preserves least upper bounds up to homotopy. Suppose that $(s,[x])$ and $(t,[y])$ are branch points of $X$, and let $(s_{0},[x_{0}]) \leq (s,[i(x)])$ and $(t_{0},[y_{0}]) \leq (t,[i(y)])$ be maximal branch points below the images of $(s,[x])$ and $(t,[y])$ in $\Gamma_{k}(Y)$. Then
\begin{equation*}
  (s_{0},[x_{0}]) \cup (t_{0},[y_{0}]) \leq (s,[i(x)]) \cup (t,[i(y)]),
\end{equation*}
so that
\begin{equation*}
  i_{\ast}(s,[x]) \cup i_{\ast}(t,[y]) \leq i_{\ast}((s,[x]) \cup (t,[y])).
\end{equation*}

Similar inequalities hold for least upper bounds with respect to the other maps that one encounters, namely $\theta_{\ast}: \Br_{k}(Y) \to \Br_{k}(X)$ and the shift map $s_{\ast}: \Br_{k}(X) \to \Br_{k}(X)$.
\medskip

\noindent
1)\ Consider the poset maps
\begin{equation*}
  \Br_{k}(X) \xrightarrow{i_{\ast}} \Br_{k}(Y) \xrightarrow{\theta_{\ast}} \Br_{k}(X).
\end{equation*}

If $(s,[x])$ is a branch point for $X$, choose maximal branch points $(s_{0},[x_{0}]) \leq (s,[i(x)]$ for $Y$, $(s_{1},[x_{1}]) \leq (s_{0}+2r,[\theta(x_{0})])$ and $(v,[y]) \leq (s+2r,[x])$ below the respective objects.

Then $\theta_{\ast}i_{\ast}(s,[x]) = (s_{1},[x_{1}])$, and there is a natural relation
\begin{equation*}
  \theta_{\ast}i_{\ast}(s,[x]) = (s_{1},[x_{1}]) \leq (v,[y]) = s_{\ast}(s,[x]) \leq (s+2r,[x]).
\end{equation*}
We therefore have a homotopy of poset maps
\begin{equation*}
  \theta_{\ast}i_{\ast} \leq s_{\ast}: \Br_{k}(X) \to \Br_{k}(X).
\end{equation*}

Note that $(s,[x]) \leq s_{\ast}(s,[x])$ since $(s,[x])$ is a branch point and $s_{\ast}(s,[x])$ is the maximal branch point below $(s+2r,[x])$. This means that the shift morphism $s_{\ast}$ is homotopic to the identity on $\Br_{k}(X)$.

The branch point $(s,[x])$ has a ``close'' shared upper bound $(s+2r,[x])$ with the element $(s_{0}+2r,[\theta(x_{0})])$, which is the image of the branch point $(s_{0},[x_{0}])$ under the poset map $\theta_{\ast}: \Gamma_{k}(Y) \to \Gamma_{k}(X)$.
\medskip

\noindent
2)\ Similarly, if $(t,[y])$ is a branch point of $Y$, then
\begin{equation*}
  i_{\ast}\theta_{\ast}(t,[y]) \leq s_{\ast}(t,[y]) \geq (t,[y])
\end{equation*}
while $s_{\ast}(t,[y]) \leq (t+2r,[y])$.

The element $(t+2r,[y])$ is a close shared upper bound for $(t,[y])$ and an element of the form $(t_{0},[i(y_{0})])$, where $(t_{0},[y_{0}])$ is a maximal branch point of $X$ below $(s+2r,[\theta(y)])$.
\medskip

\noindent
    {\bf Remark}:\ The subobject of $\Br_{k}(X)$ consisting of all branch points of the form $(s,[x])$ as $s$ varies has an obvious notion of distance on it: the distance between points $(s,[x])$ and $(t,[x])$ is $\vert t-s \vert$. The closeness referred to in constructions 1) and 2) above can be expressed in terms of such a distance. 
    \medskip

Suppose that $(s_{1},[x_{1}])$ and $(s_{2},[x_{2}])$ are branch points of $X$, let
\begin{equation*}
(s,[x]) = (s_{1},[x_{1}]) \cup (s_{2},[x_{2}]),
\end{equation*}
and write
\begin{equation*}
 (t,[y]) =  (s_{1},[i(x_{1})]) \cup (s_{2},[i(x_{2})]).
\end{equation*}
Then $(s,[i(x)])$ is an upper bound for $(s_{1},[i(x_{1})])$ and $(s_{2},[i(x_{2})])$, so that $(t,[y]) \leq (s,[i(x)])$.

The element $(t+2r,[\theta(y)])$ is an upper bound for $(s_{1},[x_{1}])$ and $(s_{2},[x_{2}])$, so that
\begin{equation}\label{eq 1xx}
  (s,[x]) \leq (t+2r,[\theta(y)]) \leq (s+2r,[x]).
\end{equation}

It follows that $s-2r \leq t \leq s$, which gives a constraint on the parameter $t$ corresponding to the least upper bound $(t,[y])$ in $\Gamma_{k}(Y)$, in terms of the least upper bound $(s,[x])$ in $\Gamma_{k}(X)$, or in $\Br_{k}(X)$. The number $2r$ is a bound on the distances between the three points in (\ref{eq 1xx}).
\medskip

If the bound $2r$ is sufficiently small, then $(s,[x])$ is the largest branch point below $(s+2r,[x])$ and $s_{\ast}(s,[x]) = (s,[x])$ in that case.
Similarly, in $\Gamma_{k}(Y)$, $s_{\ast}(t,[y]) = (t,[y])$ if  $2r$ is sufficiently small.

Recall that if $(s,[x])$ is a branch point, then $s=s_{i}$ is one of the phase shift numbers. Then $(s_{i},[x])$ is the maximal branch point below $(s_{i}+2r,[x])$ if $2r < s_{i+1}-s_{i}$.

\vfill\eject

\nocite{BlumLes}
\nocite{clusters9}

\bibliographystyle{plain}
\bibliography{spt}

\end{document}